\documentclass[a4paper,11pt,draft]{amsart}
\usepackage[latin1]{inputenc}
\usepackage{a4wide}
\usepackage{amsfonts}
\usepackage{amsmath}
\usepackage{amssymb}
\usepackage{amsthm}
\usepackage{color}
\usepackage{umlaute}
\usepackage{graphicx}
\usepackage{draftcopy}
\usepackage[normal,tight,center]{subfigure}
\newtheorem{theorem}{Theorem}[section]
\newtheorem{Corollary}[theorem]{Corollary}
\newtheorem{lemma}[theorem]{Lemma}
\newtheorem{proposition}[theorem]{Proposition}

\newtheorem{remark}[theorem]{Remark}
\numberwithin{equation}{section}
\numberwithin{equation}{section}
\numberwithin{equation}{section}

\begin{document}
\title[Ergodicity of the fiber lay down process]{Geometric Ergodicity of a hypoelliptic diffusion modelling the melt-spinning process of nonwoven materials}
\author{Martin Kolb}\footnotetext{corresponding author : M.Kolb@warwick.ac.uk}
\address{University of Warwick, Department of Statistics,
Coventry, CV4 7AL, United Kingdom}
\email{M.Kolb@Warwick.ac.uk}
\author{Mladen Savov}
\address{University of Oxford, Department of Statistics, 1 South Parks Road
Oxford, CV4 7AL, United Kingdom}
\email{mladensavov@hotmail.com}
\author{Achim W\"ubker}
\address{Universit\"at Osnabr\"uck, Fachbereich Mathematik, Albrechtstrasse 28a, 49076 Osnabr\"uck, Germany}
\email{awuebker@Uni-Osnabrueck.de}
\begin{abstract}
We analyze the large time behavior of a stochastic model for the lay-down of fibers on a conveyor belt in the production process of nonwovens.  It is shown, that under weak conditions this degenerate diffusion process is strong mixing, confirming a conjecture of Grothaus and Klar. Moreover, under some additional assumptions even geometric ergodicity is established using probabilistic tools -- described in the book of Meyn and Tweedie -- in combination with methods from stochastic analysis.  
\end{abstract}
\maketitle
\section{Introduction and Notation}
Motivated by industrial production methods  of nonwoven materials, several mathematical models of different complexity have been developed in \cite{BGKMW07}, \cite{MW06} and \cite{GKMW07} in order to accurately describe and optimize these web forming processes. In the melt-spinning process a large amount of polymer fibers are exposed to an turbulent air flow, which induces an entanglement of the different fibers and eventually the formation of a web. Production methods of this type are used in the textile, hygiene as well as the building industry. Typical products include clothing textiles, industrial filters and insulating materials.  For more details concerning the engineering background as well as the modeling framework we refer to \cite{KMW} and references therein.  In this work we are investigating the simplest of these models. This stochastic model for the fiber lay down process has been presented in \cite{GKMW07}.  It is described by a stochastic differential system modeling the image of the fiber motion on the conveyer belt under the influence of the turbulent  air flow. The interaction of the fiber motion and the air flow is formulated via the following stochastic differential equation
\begin{equation}\label{i:SDE}
\begin{split}
d\xi _t&= \tau(\alpha_t) \,dt\\
d\alpha_t &= \sigma \, dB_t - \nabla \varphi (\xi_t)\cdot \tau(\alpha_t)^{\perp}\,dt,
\end{split}
\end{equation}
where $\tau(\alpha)=\bigl(\cos(\alpha),\sin(\alpha)\bigr)^{T}$ and $\tau(\alpha)^{\perp}=(-\sin(\alpha),\cos(\alpha)\bigr)$.  Let us emphasize that it is natural to consider the driving  Brownian motion $(B_t)_{t\geq 0}$ and thus also the process $(\alpha_t)_{t \geq 0}$ as processes on the circle $\mathbb{S}_1$ or alternatively in $[0,2\pi]$ with periodic boundary conditions.  The natural state space $S$ for the SDE \eqref{i:SDE} is therefore $S=\mathbb{R}^2\times \mathbb{S}_1$. The generator $L_{\varphi}$ of this process $(X_t)_{t \geq 0}=(\xi_t,\alpha_t)_{t \geq 0}$ is of course given by
\begin{displaymath}
L_{\varphi} = \frac{\sigma^2}{2}\partial^2_{\alpha} + \cos(\alpha)\partial_{\xi_1} + \sin(\alpha)\partial_{\xi_2} - \nabla\varphi(\xi)\cdot \tau(\alpha)^{\perp}\partial_{\alpha}.
\end{displaymath}
This differential operator is obviously neither sectorial nor elliptic and therefore the usual properties are much more difficult to establish than in the standard elliptic situation. Usually even the proof of existence of an invariant distribution and the derivation of its basic properties can be very demanding.
In the case of \eqref{i:SDE} fortunately a straightforward calculation -- using the generator --  (see Proposition \ref{p:invariantstate} below as well as Theorem 3.2 in \cite{GK}) shows that  the measure $\mu$ on $S$, given by 
\begin{equation}\label{e:invariantstate}
\mu(d\xi,d\alpha) := \frac{1}{N}e^{-\varphi(\xi)}\,d\xi\,d\alpha
\end{equation}
is a promising candidate for an invariant distribution for the process $(\xi_t,\alpha_t)_{t \geq 0}$, once
\begin{displaymath}
N:= \int _S e^{-\varphi(\xi)}\,d\xi\,d\alpha < \infty.
\end{displaymath}
In the recent paper \cite{GK} M. Grothaus and A. Klar use analytic methods close to those presented in \cite{OT} in order to show that 
\begin{equation}\label{e:GKresult}
\biggl\| \frac{1}{t}\int_0^tf(X_s)\,ds-\mathbb{E}_{\mu}\bigl[ f \bigr]\biggr\|_{L^2(\mathbb{P}_\mu)} \leq \frac{1}{\sqrt{c}}\biggl(\frac{2}{t}+ \frac{c_{\sigma,\varphi}}{\sqrt{t}}\biggr)\bigl\| f-\mathbb{E}_{\mu}[f]\bigr\|_{L^2(\mu)}.
\end{equation}
Actually, Grothaus and Klar are even able to provide a rather explicit description of the constants involved. On the one hand this result is obviously quite satisfactory as it gives explicit upper bounds on the rate of convergence. On the other hand it only makes an assertion about the large time behavior of time averages and not of the process itself. The authors of \cite{GK} conjecture that for smooth potentials the process $(X_t)_{t \geq 0}$ is in fact also strong mixing in the sense of 
\begin{displaymath}
\lim_{t \rightarrow \infty} \int_S \bigl| \mathbb{E}_x[f(X_t)]-\mathbb{E}_{\mu}[f]  \bigr|^2\,d\mu(x)= 0.
\end{displaymath}
We will prove this for a larger class of potentials $\varphi$. A rather immediate question is, whether one has a geometric rate of convergence. This will be answered affirmatively in section \ref{s:geomrate} under further assumptions on $\varphi$ in the sense that under appropriate conditions Theorem \ref{t:main} implies
\begin{displaymath}
-\lim_{t \rightarrow \infty}\frac{1}{t}\log \bigl\| \mathbb{P}_x\bigl(X_t \in \cdot\bigr) -\mu \bigr\|_{TV}> 0
\end{displaymath}
for every $x\in S$. In contrast to \cite{GK} we use mainly probabilistic techniques and do not rely on Dirichlet form and semigroup methods. A geometric rate of convergence is usually associated to the spectral gap property in some suitable chosen norm (see e.g. \cite{H}, \cite{MTBook}, \cite{MT93} and \cite{Wu04})  and therefore one might try to prove this using methods from spectral theory. The non-ellipticity of this process makes such an approach rather difficult and we do not know, whether recently developed analytic tools (see \cite{Vil}) might be applied in the situation at hand.  In any case, our approach only relies on rather basic arguments much more elementary than those needed in the theory of hypocoercivity. Our approach has the potential to extend to the situation of a moving conveyor belt (see e.g. equation (6) in \cite{KMW}), where the existence of a stationary distribution is not known apriori. Last but not least we believe that our stochastic analysis shows once more the power of the coupling approach outlined in \cite{MTBook} and offers additional insight into the probabilistic mechanism underlying the fiber lay down process. 

The plan for this paper  is the following : In the section 2.1 we present existence, uniqueness as well as some basic properties of the solution of the solutions to \label{i:SDE}. Some of these are known but we give a complete account in order to be independent from \cite{GK}, where it is made use of the theory of generalized Dirichlet forms; others contain new results such as e.g. the strong Feller property proved in Proposition \ref{p:Feller}. In section 2.2 we prove ergodicity under weak conditions on $\varphi$ and in section 3 we show geometric ergodicity under stronger assumption on $\varphi$. 
\section{First Properties of The Fiber Lay Down Process}
We start by establishing existence and uniqueness of solutions of \eqref{i:SDE} in the weak sense. 
\begin{proposition}\label{p:Girsanov}
Assume that $\varphi \in C^1(\mathbb{R}^2)$. For every initial state $x=(\xi_0,\alpha_0) \in S$ the stochastic differential equation \eqref{i:SDE} has a unique weak solution $\mathbb{P}_x$, which is non-explosive. Let us denote by $\mathbb{Q}_x$ the solution for the case $\varphi=0$ then for every $t>0$ we have $\mathbb{P}_x\restriction \mathcal{F}_t << \mathbb{Q}_x \restriction \mathcal{F}_t$ and the following Girsanov formula holds:
\begin{equation}
\frac{d\mathbb{P}_x}{d\mathbb{Q}_x}\restriction \mathcal{F}_t = exp\biggl(- \int_0^t \nabla \varphi(\xi_s)\cdot \tau(\alpha_s)\,d\alpha_s- \frac{\sigma^2}{2}\int_0^t \bigl|\nabla\varphi(\xi_s)\cdot \tau(\alpha_s)\bigr|^2\,ds \biggr)
\end{equation}
\end{proposition}
\begin{proof}
The proof essentially follows ideas outlined in \cite{Wu01} and is devided into two simple steps:\\
\textit{Step 1:} We first look at the situation $\varphi = 0$. In this case the stochastic differential equation is well known to have (explicit) solutions, which are unique in the weak as well as in the pathwise sense. The solution forms a conservative strong Markov process $(\xi^0_t,\alpha^0_t)_{t \geq 0}$ having the Feller property in the sense that for every bounded measurable $g: S \rightarrow \mathbb{R}$ and $t > 0$ the function 
\begin{displaymath}
S \ni x \mapsto \int g((\xi^0_t,\alpha^0_t))\,d\mathbb{Q}_x
\end{displaymath}
is continuous.\\ 
\textit{Step 2:} We now look at the case $\varphi \in C^1(\mathbb{R}^2)$ assuming additionally that $\nabla \varphi$ is bounded.  In this situation we rely on the Cameron-Martin-Girsanov formula as stated e.g. as Theorem 9.1 in \cite{PinskyBook}.  According to this theorem or respectively Corollary 9.3 in \cite{PinskyBook} we conclude that the martingale problem associated to the operator $L_{\varphi}$ has a unique solution $(\mathbb{P}_x)_{x \in S}$ satisfying 
\begin{equation}\label{e:Girsanov}
\mathbb{P}_x  = \exp\biggl(- \int_0^t \nabla \varphi(\xi^0_s)\cdot \tau(\alpha^0_s)\,d\alpha^0_s- \frac{\sigma^2}{2}\int_0^t \bigl|\nabla\varphi(\xi^0_s)\cdot \tau(\alpha^0_s)\bigr|^2\,ds \biggr)\mathbb{Q}_x,
\end{equation}
where we denote by $\mathbb{Q}_x$ the law of the Markov process $(\xi^0_t,\alpha^0_t)_{t \geq 0}$ from Step 1 with the initial condition $(\xi^0_0,\alpha^0_0)=x$. In particular the exponential in  \eqref{e:Girsanov} is a true martingale.\\
\textit{Step 3:} We now drop the condition of boundedness of $\nabla \varphi$.
\\
\\
Let us denote by 
\begin{displaymath}
\sigma_R=\inf\lbrace t \geq 0 \mid X_t=(\xi_t,\alpha_t) \notin  \overline{B(0,R)}\times \mathbb{S}_1\rbrace
\end{displaymath}
the first exit time from the compact set $ \overline{B(0,R)} \times \mathbb{S}_1$.
We now show that for every $x \in {S}$ and $t>0$
\begin{equation}\label{e:nonexplosion}
\lim_{R \rightarrow \infty}\mathbb{P}_x\bigl(\sigma_R > t \bigr)=1.
\end{equation}
Assume for a moment that \eqref{e:nonexplosion} holds true, then 
\begin{equation}
\int M_t \, d\mathbb{Q}_x \geq  \int_{\lbrace \sigma_R > t \rbrace} M_{t \wedge \sigma_R}\,d\mathbb{Q}_x =\mathbb{P}_x\bigl(\sigma_R > t\bigr) \rightarrow 1
\end{equation}
as $R \rightarrow \infty$. This shows that the local martingale $(M_t)_{t \geq 0}$ is a true martingale. 
\\
\\
In order to prove \eqref{e:nonexplosion} we rely on the Liapunov method as presented e.g. in section 6.7 \cite{PinskyBook}. Let us set $\Psi(\xi,\alpha)=|\xi|^2+1$, then for all $\xi=(\xi_1,\xi_2)$ with $\xi_1,\xi_2\geq 1$ 
\begin{displaymath}
L_{\varphi}\Psi(\xi,\alpha) = \cos(\alpha)\partial_{\xi_1}\Psi(\xi,\alpha) + \sin(\alpha)\partial_{\xi_2}\Psi(\xi,\alpha) \leq 2\bigl(\xi_1 + \xi_2\bigr)\leq 2\Psi(\xi,\alpha).
\end{displaymath}
Consequently, 
\begin{displaymath}
\bigl( e^{-2(t\wedge\sigma_R)}\Psi(X_{t \wedge\sigma_R})\bigr)_{ t \geq 0}
\end{displaymath}
defines a $\mathbb{P}_x$-supermartingale and therefore
\begin{equation}\label{e:supermartingale}
\mathbb{E}_x\bigl[e^{-2(t \wedge\sigma_R)}\Psi (X_{t\wedge\sigma_R})\bigr] \leq \Psi(x).
\end{equation}
Using the fact that $\Psi(X_{\sigma_R}) \geq R^2+1$ on $\lbrace \tau_R < \infty\rbrace$ we get
\begin{displaymath}
\mathbb{P}_x\bigl(\sigma_R \leq t\bigr) \leq \frac{e^{2t}}{R^2+1}\mathbb{E}_x\bigl[\mathbf{1}_{\lbrace \sigma_R \leq t\rbrace}e^{-2(\sigma_R\wedge t)}\Psi( X_{t\wedge\sigma_R})\bigr] \leq \frac{e^{2t}}{R^2+1}\Psi(x),
\end{displaymath}
which implies \eqref{e:nonexplosion}.
\end{proof}
\begin{remark}
The fact that the process $(X_t)_{t \geq 0}$ does not explode is more or less obvious due to the fact that the 'velocity' in the $\xi$-component is bounded by one. Thus the process can not reach infinity in finite time. This is reflected in the elementary choice of the Liapunov function and the fact that the $\alpha$-component does not play any role here. We included the proof here as the Liapunov function method  will most probably be needed in the analysis of more complex models of the fiber lay down process. 
\end{remark}
\subsection{Feller and Recurrence Properties}
We now establish smoothing properties of the stochastic process $(\xi_t,\alpha_t)_{t\geq 0}$. These are essential in order to prove the mixing properties in the next section. Again our argument is perturbation theoretic in the sense, that we look for conditions, which allow to transfer these properties known to hold in the case $\varphi = 0$ to the general situation. Our arguments are again inspired by the corresponding ones in \cite{Wu01} (see also \cite{MS00} for the use of Girsanov's formula in the derivation of Feller properties).
\begin{proposition}\label{p:Feller}
Assume that $\varphi \in C^1(\mathbb{R}^2)$ and let $(P_t(x,dx'))_{t \geq 0}$ denote the transition probability kernels of the Markov process $((X_t)_{t \geq 0},(\mathbb{P}_x)_{x\in S})$. Then for every $t >0$ and $x \in S$ one has
\begin{displaymath}
P_t(x,dx')=p_t(x,x')dx',\,\quad p_t(x,x')\geq 0
\end{displaymath}
and 
\begin{displaymath}
S \ni x \rightarrow p_t(x,\cdot\,) \in L^1(S)
\end{displaymath}
is continuous. In particular, $P_t$ is strong Feller for each $t>0$. 
\end{proposition}
\begin{proof}
The proof very closely follows Proposition 1.2 in \cite{Wu01}. Consider the case $\varphi=0$ first. In this case the generator is obviously hypoelliptic and by H\"ormander's theorem (see e.g. chapter 7 in \cite{StrPDE}) we conclude that there exists $q_t \in C^{\infty}(S \times S)$ such that the transition probability $Q_t(x,dx')$ of $\mathbb{Q}$ satisfies
\begin{displaymath}
Q_t(x,dx') = q_t(x,x')dx'
\end{displaymath} 
Since $x \rightarrow \mathbb{Q}_x$ is continuous with respect to the weak convergence of probability measures we conclude the strong Feller property of $Q_t$ for every $t>0$. 

Now let $\varphi$ satisfy the conditions of Proposition \eqref{p:Girsanov}. Then according to the Girsanov formula established in the proof of this result we conclude that there exists $p_t(x,x')$ such that $P_t(x,dx')=p_t(x,x')dx'$ for every fixed $t>0$. As in \cite{Wu01} we conclude using a result of Revuz that it suffices to show, that $P_tf \in C_b(S)$ for every bounded and measurable $f:S\rightarrow \mathbb{R}$ and $t>0$.  

Let $(B_t)_{t\geq 0}$ be a Brownian motion in $[0,2\pi]$ with periodic boundary conditions defined on the filtered probability space $(\Omega,\mathcal{F},(\mathcal{F}_t)_{t\geq 0},\mathbb{Q})$.  We denote by  $(X^0_s(x))_{s\geq 0}$ the strong solution of \eqref{i:SDE} with $\varphi=0$ and with initial condition $X_0=x$, i.e. 
\begin{displaymath}
\xi_t=\begin{pmatrix} \int_0^t\cos(B_s)\,ds \\ \int_0^t\sin(B_s)\,ds \end{pmatrix}\,\text{  and  } \,\alpha_t=B_t.
\end{displaymath}
Then we have
\begin{equation}\label{e:girsform}
P_tf(x) = \mathbb{E}^{\mathbb{P}_x}\bigl[f(X_t)\bigr] = \mathbb{E}^{\mathbb{Q}}\bigl[ f(X^{0}_{t}(x))M_t(x)\bigr],
\end{equation}
where $M_t(x)$ is the exponential martingale given by the Girsanov formula in Proposition \ref{p:Girsanov}, but with $(\xi_s,\alpha_s)_{s\geq 0}$ substituted by the $(X^0_s(x))_{s\geq 0}$. Now let $(x_n)_{n\in \mathbb{N}} \subset S$ denote a sequence converging to $x \in S$. Then we have $(X^0_s(x_n))_{s\geq 0} \rightarrow (X^0_s(x))_{s\geq 0}$ locally uniformly in $\mathbb{Q}$--probability and by convergence properties of the stochastic integral we have $M_t(x_n) \rightarrow M_t(z)$ in $\mathbb{Q}$--probability. As for $n\in \mathbb{N}$ the processes $(M_t(z_n))_t$ and $(M_t(z))_t$ are martingales we conclude that $\mathbb{E}^{\mathbb{Q}}[M_t(z_n)]=1=\mathbb{E}^{\mathbb{Q}}[M_t(z)]$. Therefore according to Scheff\'e's Lemma we have 
\begin{equation}\label{e:l1convmart}
M_t(z_n) \rightarrow M_t(z) 
\end{equation}
in $L^1(\mathbb{Q})$ as $n \rightarrow \infty$. 

It remains to show that $f(X^0_t(x_n))$ converges to $f(X^0_t(x))$ as $n \rightarrow \infty$ in $\mathbb{Q}$-probability. Observe that together with \eqref{e:girsform} and \eqref{e:l1convmart} this completes the proof of 
\begin{displaymath}
\lim_{n \rightarrow \infty}P_tf(x_n)=P_tf(x).
\end{displaymath}
Let $\tilde{q}_t(\cdot,\cdot)$ be the density of $\mathbb{Q}\bigl( X^0_t(x)\in \,\cdot\,\bigr)$ with respect to $\mu$. From the strong Feller property of $(X_t^0(x))_{t \geq 0}$ and another application of Scheff\'e's lemma one deduces that the sequence $(q_t(x_n,\dot))_{n \in \mathbb{N}}$ is uniformly integrable in $L^1(\mu)$. Hence for any $\varepsilon > 0$ there exists $\delta>0$ such that for all measurable $A \subset S$,
\begin{displaymath}
\mu(A)< \delta \Rightarrow \int _A\tilde{q}_t(x_n,x)\mu(dx)=\mathbb{Q}\bigl( X^0_t(x_n)\in A \bigr)<\varepsilon\text{   for all $n$ and   }\mathbb{Q}\bigl( X^0_t(x)\in A \bigr)<\varepsilon.
\end{displaymath} 
By Egorov's theorem there exists a compact set $K\subset S$ such that $\mu(K^c)<\delta$ and $f\restriction K$ is uniformly continuous, i.e. for any $\eta>0$ there is $\delta'>0$ such that $|z-z'|<\delta', z,z'\in K \Rightarrow |f(z)-f(z')|<\eta$. There exists $N_0>0$ such that for $n\geq N_0$ 
\begin{displaymath}
\mathbb{Q}\bigl( \bigl| X^0_t(x_n)-X^0_t(x)|\geq \delta' \bigr)<\varepsilon
\end{displaymath}
and thus
\begin{equation*}
\begin{split}
\mathbb{Q}\bigl(\bigl| f(X^0_t(x_n))-f(X^0_t(x))\bigr|>\eta \bigr) &\leq \mathbb{Q}\bigl( \bigl| X^0_t(x_n)-X^0_t(x)|\geq \delta' \bigr)+ \mathbb{Q}\bigl(X^0_t(x_n) \notin K\bigr)\\
&+\mathbb{Q}\bigl(X^0_t(x) \notin K\bigr)\leq 3 \varepsilon.
\end{split}
\end{equation*}
As observed earlier this completes the proof.
\end{proof}
In order to make the paper independent from the approach used by Grothaus and Klar, which used the machinery of generalized Dirichlet forms we add
\begin{proposition}\label{p:invariantstate}
If $\varphi$ satisfies $e^{-\varphi} \in L^1(\mathbb{R}^2)$ then the measure $\mu$ introduced in \eqref{e:invariantstate} defines an invariant distribution for the process $(X_t)_{t \geq 0}$.
\end{proposition}
\begin{proof}
Again we emphasize that the result is proved in \cite{GK} by different means. As mentioned in the introduction a direct calculation shows that the dual $\tilde{L}$ of $L$ is given by
\begin{displaymath}
\tilde{L} = \frac{\sigma^2}{2}\partial_{\alpha}^2 - \cos(\alpha)\partial_{\xi_1} - \sin(\alpha)\partial_{\xi_2} + \nabla\varphi(\xi) \cdot \begin{pmatrix} -\sin(\alpha) \\ \cos(\alpha) \end{pmatrix}\partial_{\alpha} + \nabla\varphi(\xi)\cdot \begin{pmatrix} -\cos(\alpha)  \\ -\sin(\alpha) \end{pmatrix} 
\end{displaymath}
and that $\tilde{L}\mu=0$. Setting $\tilde{L}^{\mu}=\frac{1}{\mu}\tilde{L}\bigl(\mu\,\cdot)$, where by a slight abuse of notation $\mu$ is also used to denote the density of the measure  \eqref{e:invariantstate}, we get 
\begin{displaymath}
\tilde{L}^{\mu}= \frac{\sigma^2}{2}\partial_{\alpha}^2 - \cos(\alpha)\partial_{\xi_1} - \sin(\alpha)\partial_{\xi_2} + \nabla\varphi(\xi) \cdot \begin{pmatrix} -\sin(\alpha) \\ \cos(\alpha) \end{pmatrix}\partial_{\alpha}.
\end{displaymath}
Completely analogous to the proof of Proposition \ref{p:Girsanov} and Proposition \ref{p:Feller} one constructs a conservative Feller process solving the martingale problem associated to $\tilde{L}^{\mu}$. From the conservativity we conclude by Theorem 8.5 of \cite{PinskyBook} that $\mu$ is in fact an invariant distribution. Observe that Theorem 8.5 in \cite{PinskyBook} -- though formulated for elliptic diffusion -- holds true in our situation with the same proof.
\end{proof}
The next result is concerned with the positivity properties of the process. 
\begin{proposition}\label{l:irred}
Let $\varphi \in C^1(\mathbb{R}^2)$ and let $\bigl(\Omega,\mathcal{F},(\mathcal{F}_t)_{t \geq 0}, (\mathbb{P}_x)_{x \in S},(X_t)_{t \geq 0}\bigr)$ denote the unique weak solution of \eqref{i:SDE} on the path space $\Omega=C([0,\infty),S)$. If $B\subset S$ is an open subset of the state space and if $x \in S$ then there exists $T \geq 0$ such that for $t \geq T$
\begin{displaymath}
\mathbb{P}_x\bigl(X_t \in B\bigr)>0.
\end{displaymath} 
Moreover, every skeleton Markov chain is irreducible. 
\end{proposition}
\begin{proof}
The irreducibility of skeletons follows from the previous assertions together with Proposition \ref{p:Feller}  and section 6.1.2 in \cite{MTBook}.

We first show the assertion for the case $\varphi = 0$. Let us denote by $B(\tilde{\xi},\varepsilon)$ the ball of radius $\varepsilon>0$ and center $\tilde{\xi}$. It suffices to prove that for a set $B$ of the form $B(\tilde{\xi},\varepsilon)\times (\tilde{\alpha}-\varepsilon,\tilde{\alpha}+\varepsilon)$ with $\tilde{\xi} \in \mathbb{R}^2$, $\tilde{\alpha} \in \mathbb{S}_1$ and $\varepsilon > 0$ there exists $T \geq 0$ such that for $t\geq T$ 
\begin{displaymath}
\mathbb{Q}_x\bigl(X_t \in B\bigr)>0,
\end{displaymath}
where we have set $x=(\tilde{\xi},\tilde{\alpha})$. Assume we know that there exists a continuous path $(\bar{\alpha}_s)_{s \geq 0}$ with values in $\mathbb{S}_1$ such that $\bar{\alpha}_t \in (\tilde{\alpha} - \varepsilon,\tilde{\alpha} + \varepsilon)$ and 
\begin{equation}\label{e:controlproblem}
\bar{\alpha}_t \in (\tilde{\alpha} - \varepsilon,\tilde{\alpha} + \varepsilon)\,\text{  and  }\, \biggl(\int_0^t\cos(\bar{\alpha}_s)\,ds, \int_0^t\sin(\bar{\alpha}_s)\,ds\biggr) \in B(\tilde{\xi},r)
\end{equation}
then there exists obviously $\delta>0$ such that for every path $(\alpha'_s)_{s \geq 0}$ with $\sup_s|\bar{\alpha}_s - \alpha_s'| < \delta$ we also have $\alpha'_t \in (\tilde{\alpha} - \varepsilon,\tilde{\alpha} + \varepsilon)$ as well as
\begin{displaymath}
\biggl(\int_0^t\cos(\alpha_s')\,ds, \int_0^t\sin({\alpha}_s')\,ds\biggr) \in B(\tilde{\xi},r).
\end{displaymath}
Since the support of the Wiener measure on $C([0,\infty),\mathbb{R}^2)$ coincides with $C([0,\infty),\mathbb{R}^2)$ we conclude that 
\begin{displaymath}
 \mathbb{Q}_x\bigl( X_t \in B(\tilde{\xi},r)\times (\tilde{\alpha} - \varepsilon,\tilde{\alpha} + \varepsilon)\bigr)\geq \mathbb{P}^{BM}_{\alpha_0=\tilde{\alpha}}\bigl(\lbrace\alpha' \mid \sup_{s \leq t}|\alpha'_s-\bar{\alpha}_s|\leq \delta\bigr\rbrace\bigr)>0,
\end{displaymath}
where $\mathbb{P}_{\alpha_0}^{BM}$ denotes the law of a two-dimensional Brownian motion on $C([0,\infty),\mathbb{R}^2)$, starting at $\alpha_0$. 
Thus we only have to construct $(\bar{\alpha})_s)_{s\geq 0}$ having the properties \eqref{e:controlproblem}. This can certainly be deduced from control theory but we prefer a direct hands on argument. We choose the angle $\hat{\alpha}$ such that  the vector $(\cos ( \hat{\alpha} ),\sin( \hat{\alpha} ) )^{T}$ is a positive multiple of the vector $\tilde{\xi}-\xi_0$, say $(\cos ( \hat{\alpha} ),\sin( \hat{\alpha} ) )^{T} = c(\tilde{\xi}-\xi_0)$. Let $T=1/c$ and $T>t_2>t_1>0$ and let $(\alpha^{t_1,t_2}_t)_{0\leq t\leq t_1}$ be defined by $\alpha^{t_1,t_2}_t = (1-t/t_1) {\alpha}_0 + t/t_1\hat{\alpha}$. For $t \in (t_1,t_2)$ we set $\alpha^{t_1,t_2}_t=\hat{\alpha}$ and for $t_2\leq t \leq T$ we set $\alpha^{t_1,t_2}_t = \bigl(1-\frac{t-t_2}{T-t_2}\bigr) \hat{\alpha} + \frac{t-t_2}{T-t_2}\tilde{\alpha}$. By choosing $t_1$ and $T-t_2$ sufficiently small one gets the properties \eqref{e:controlproblem} for $t=T$ and the extension to $t > T$ can be done by extending $\alpha^{t_1,t_2}$ in an appropriate way similiar to what has just been done. 

This establishes the assertion for $\varphi = 0$. The general case then follows from the Girsanov formula \eqref{e:Girsanov}.
\end{proof}
We are now using ideas from the proof of Theorem 1 of \cite{St94} in order to prove the Harris recurrence of the fiber lay down process. Recall that $(X_t)_{t \geq 0}$ is called Harris recurrent if there exists a finite measure $\gamma$ on the Borel $\sigma$-algebra over the state space such that for every Borel set $A$ with $\gamma(A)>0$ we have 
\begin{displaymath}
\mathbb{P}_x\biggl(\int_0^{\infty}\mathbf{1}_A(X_s)\,ds < \infty\biggr)=0
\end{displaymath}
for every $x \in S$. The definition of Harris recurrent for a Markov chain is analogous, the integral has to be replaced by a sum. $(X_t)_{t \geq 0}$ is Harris recurrent if and only if the skeleton chain $X^h_n=(X_{hn})_{n\in \mathbb{N}_0}$ is Harris recurrent for some $h>0$. 
\begin{proposition}\label{p:Harris}
The fiber-lay down process $(X_t)_{t \geq 0}$ is Harris recurrent.
\end{proposition}
\begin{proof}
The proof closely follows the proof of Theorem 1 in \cite{St94} and is included here for reasons of completeness. We first note, that the invariant distribution $\mu$ is ergodic in the sense, that $\mu$ is the unique invariant distribution. It follows from Proposition \ref{l:irred} that two different invariant distributions have to be mutually singular and Proposition \ref{p:Feller} implies that invariant distributions have to be absolutely continuous with respect to the Lebesgue measure on $S$. As the support of $\mu$ is the full state space the probability measure $\mu$ must be unique. Thus the individual ergodic theorem implies that  there exists a $\mu$-null set $N\subset S$ such that for any $x \in S\setminus N$ and for any Borel set $A \subset S$
\begin{displaymath}
\lim_{n \rightarrow \infty}\frac{1}{n}\sum_{i=0}^{n}{P}_{hi}(x,A)\rightarrow \mu(A).
\end{displaymath}
Since $N$ is also a Lebesgue null set we conclude by Proposition \ref{p:Feller} that $P_t(x,N)=0$ for every $t>0$ and every $x \in S$ and therefore for $x \in S$
\begin{displaymath}
n^{-1}\sum_{i=1}^{n}P_{hi}(x,A)=\int_Sn^{-1}\sum_{i=0}^{n-1}P_{hi}(y,A)P_h(x,dy)=\int_{S\setminus N}n^{-1}\sum_{i=0}^{n-1}P_{hi}(y,A)P_{h}(x,dy)\rightarrow \mu(A)
\end{displaymath}
as $n\rightarrow \infty$.
Hence
\begin{displaymath}
\lim_{n \rightarrow \infty}\frac{1}{n}\sum_{i=0}^{n}{P}_{hi}(x,A)\rightarrow \mu(A).
\end{displaymath}
for every $x \in S$ and every Borel set $A$.  Thus for any Borel set $A$ with $\mu(A)>0$ we have for every $x \in S$
\begin{displaymath}
\sum_{n=0}^{\infty}P_{nh}(x,A) = \infty, 
\end{displaymath}
showing the recurrence of the skeleton chain. Thus according to 9.1.5 in \cite{MTBook} we can decompose the state space into disjoint sets $H$ and $N$, i.e. $S=H\cup N$, such that $H$ is a maximal Harris set and $N$ is transient such that $\mu(N)=0$ and $H=\lbrace y \in S\mid \mathbb{P}_y\bigl(\kappa_H < \infty\bigr)=1\rbrace$ and $(X_{nh})_{n}$ restricted to $H$ is Harris, where $\kappa_H = \inf\lbrace n \mid X_{nh} \in H \rbrace$. Since $N$ is also a Lebesgue null set we conclude that $P_{nh}(x,N)=0$ for every $x \in S$ and $n\geq 1$. Thus, from the definition of $H$, $N$ must be empty and $(X_{hn})_{n\in \mathbb{N}}$ Harris. Therefore the skeleton chain $(X_{hn})_{n\in \mathbb{N}_0}$ is Harris recurrent and thus the same is true for $(X_t)_{t \geq 0}$. 
\end{proof}
The above results allows us to deduce the following corollary. First recall that a set $C \subset S$ is called \textit{small} for $(X_t)_{t\geq0}$, if there exists a  $T>0$ and a non-trivial measure $m$ on the Borel sets of $S$, such that 
\begin{displaymath}
\forall x \in C : P_T(x,\,\cdot\,) \geq m_T(\cdot).
\end{displaymath}
The Markov transition kernel $P_t$ is called \textit{aperiodic} if for some for some small set $C \subset S$ there exists $T>0$ such that for every $x\in C$ and $t \geq T$ one has $P_t(x,C)>0$
\begin{Corollary}\label{C:small}
All compact subsets of the state space $S$ are small sets for the Markov process $(X_t)_{t \geq 0}$ and the Markov process $(X_t)_{t \geq 0}$ is aperiodic.
\end{Corollary}
\begin{proof}
According to Propositions \ref{p:Feller} and \ref{p:Harris}  together with Lemma \ref{l:irred} we know that $(X_t)_{t \geq 0}$ is an irreducible strong Feller process. Thus according to Theorem 4.1 in \cite{MT93a} compact sets are petite sets and according to Proposition 4.1 in \cite{MT93a} petite sets are small for $(X_t)_{t \geq 0}$. The second assertion now follows from the fact that open sets are small and Proposition \ref{l:irred}.
\end{proof}
\subsection{Strong Mixing Properties}
We now move on and prove ergodicity of the fiber lay down process, more precisely we show that it is strongly mixing in the sense that for any initial distribution $\nu$ on $S$ 
\begin{equation}\label{e:mixI}
\lim_{t \rightarrow \infty}\bigl\| \mathbb{P}_{\nu}\bigl(X_t \in \,\cdot\, \bigr)-\mu \bigr\|_{TV}=0 
\end{equation}
as well as 
\begin{equation}\label{e:mixII}
\lim_{t \rightarrow \infty}\biggr\| \mathbb{E}_{(\xi_0,\alpha_0)}\bigl[ f(\xi_t,\alpha_t) \bigr]-\int_Sf(x)\,\mu(dx) \biggr\|_{L^2(\mu)}=0
\end{equation}
for every $f \in L^2(\mu)$. Observe that on the one hand these assertions are stronger than the one in \eqref{e:GKresult}, since they do not require to look at time averages and only rely on very weak assumptions concerning $\varphi$. On the other hand they are much weaker than  \eqref{e:GKresult} as the do not provide any information concerning the rate of convergence. 
\begin{theorem}\label{t:mixing}
Assume that $\varphi$ satisfies the conditions of Proposition \ref{p:Feller}. Then the assertions \eqref{e:mixI} and \eqref{e:mixII} hold true. 
\end{theorem}
\begin{proof}
As is well-known strong mixing will be a rather direct consequence of Doob's theorem once irreducibility is established. More precisely, assertion \eqref{e:mixI} follows e.g. from Theorem 1 of \cite{St94}, where one has to observe that the assumption 'for a $h>0$ the transition measures $P_h(x,\cdot )$ are for $x \in S$ equivalent' is only used in order to  prove Harris recurrence for the skeleton chain $(X_{nh})_{n\in \mathbb{N}_0}$. For the fiber-lay down process $(X_t)_{t \geq 0}$ this has been shown in Proposition \ref{p:Harris}.  Following the proof of Theorem 1 of \cite{St94} we can conclude that the skeleton chain  $(X_{nh})_{n\in \mathbb{N}_0}$ converges as $n\rightarrow $ in total variation to $\mu$ and according to Theorem 3.7 in \cite{Rev} this implies that 
\begin{displaymath}
\lim_{t \rightarrow \infty}\bigl\|\mathbb{P}_x\bigl(X_t\in \,\cdot \,)-\mu\bigr\|_{TV}=0.
\end{displaymath}
for ever $x \in S$. 
In order to prove assertion \ref{e:mixII} we rely on Korollar 3.11 in \cite{Gre82} (see also Theorem 4.5 in \cite{GN11}). Since according to Propostion \ref{p:Feller} we know that for every $t >0$ the transition kernels $P_t$ consist of integral operators having a non-negative kernel $p_t(\cdot,\cdot)$. Thus Korollar 3.11 implies that for any $f \in L^p(\mu)$ ($1\leq p<\infty$)
\begin{displaymath}
\lim_{t \rightarrow \infty}\biggl\|P_tf-\int_Sf\,d\mu\biggr\|_{L^p(\mu)}=0,
\end{displaymath}
completing the proof of the theorem. 
\end{proof}

\section{Geometric Ergodicity}\label{s:geomrate} 
Until now it seems to be unknown, whether the rate of convergence towards the invariant distribution is geometric. This is at least strongly suggested by numerical experiments presented in \cite{GK} for the cases $\varphi(\xi) = |\xi|^2$ and $\varphi(\xi)=|\xi|$. In this section we will investigate the question of geometric ergodicity.  From the theory of Markov chains it is known that geometric ergodicity is strongly connected to exponential integrability of hitting times of compact sets. Indeed, the existence of a Lyapunov function -- also called geometric drift conditions -- are essentially equivalent to the exponential integrability of hitting times of a compact set, which is a small set by Corollary \ref{C:small} (see \cite{Wu04} and chapter 15 in \cite{MTBook}).  Since the angle component $\alpha_t$ of the process moves in a compact state space the problem consists in showing, that the $\xi_t$-component returns to compact sets fast enough.  Considering the first hitting time of a ball in $\mathbb{R}^2$ of radius $R$ around $0$ it is tempting to look at the process $(|\xi_t|)_{t \geq 0}$. Assuming now that $\varphi$ is spherically symmetric it is convenient to write the SDE \eqref{i:SDE} in polar coordinates, i.e. we set $\xi=(r \cos(\psi),r\sin(\psi))$ and $\beta=\alpha - \psi$. According to  formula (13) in \cite{KMW} (see also p. 5 in \cite{GKMW07}) this leads to the following system of stochastic differential equations
\begin{equation}\label{e:radialSDE}
\begin{split}
d r_t&= \cos(\beta_t)\,dt \\
d\beta_t &= \bigl(b(r_t) - \frac{1}{r_t}\bigr)\sin(\beta_t)\,dt + \sigma \,dB_t\\
d\psi_t &= \frac{\sin(\beta_t)}{r_t},
\end{split}
\end{equation}
where according to \cite{KMW} $b(r) = \varphi'(r)$. Equation \eqref{e:radialSDE} holds in the region, where $\varphi$ is radial.

This system has the advantage that the equation for $(\psi_t)_{t \geq 0}$ is decoupled from the other $(r_t,\beta_t)_{t \geq 0}$-process and that the equation for the radial part looks quite handy. Let us assume that the following assumption holds:
\begin{itemize}
\item[\textbf{A)}]There exists $R$ and a constant $c>0$ such that $b(r) - \frac{1}{r}\geq c$ for every $r \geq R$. 
\end{itemize}
This assumption is needed in order to make use of comparison arguments with the following SDE:
\begin{equation}\label{e:radialSDEcomp}
\begin{split}
d s_t&= \cos(\delta_t)\,dt \\
d\delta_t &= c\sin(\delta_t)\,dt + \sigma \,dB_t
\end{split}
\end{equation}
At this point one has to note, that we are interested in the first hitting time of the process $(\beta_t,r_t)_{t \geq 0}$ of compact sets of the form $\mathbb{S}_1 \times [0,R]$ and therefore the singular behavior near $0$ will be of no concern here.

Obviously it is not necessary for us to assume that $\varphi$ is radially symmetric on the full $\mathbb{R}^2$. Let us thus weaken this condition slightly.  We call a potential $\varphi$ \textit{eventually radial}, if there exists a ball $B(0,l)$ with radius $l \geq 0$ and centre $0$, such that $\varphi \restriction \mathbb{R}^2 \setminus B(0,l)$ is radially symmetric. 
\begin{proposition}\label{p:expintegrability}
Assume that $\varphi \in C^2(\mathbb{R}^2) $ is eventually radial and assume that Assumption A) is satisfied. Then there exists $\lambda > 0$ such that for every $(r_0,\beta_0) \in (0,\infty)\times [0,2\pi)$
\begin{displaymath}
\mathbb{E}_{(r_0,\beta_0)}\bigl[ e^{\lambda \tau_R} \bigr] < \infty,
\end{displaymath}
i.e. the random variable $\tau_R$ has exponential moments of order $\lambda>0$. 
\end{proposition}
Let us for a moment assume that Proposition \ref{p:expintegrability} holds true and let us see, how this implies the following theorem, which constitutes the main result of this paper.  We formulate the theorem for the case of a gradient drift $\nabla \varphi$ in equation \eqref{i:SDE} but we explain in Remark \ref{r:final} that the gradient form of the drift is not important for the following theorem.
\begin{theorem}\label{t:main}
Assume that  $\varphi \in C^2(\mathbb{R}^2)$ is eventually radial and assume that Assumption A) is satisfied. Then there exists a constant $v > 0$ and a strictly positive function $M:S\rightarrow (0,\infty)$ such that 
\begin{displaymath}
\forall t > 0: \bigl\| \mathbb{P}_{(\xi_0,\alpha_0)}\bigl( (\xi_t,\alpha_t\bigr) \in \cdot \,\bigr) - \mu \bigr\|_{TV}\leq M(r_0,\beta_0) e^{-v t}
\end{displaymath}
\end{theorem}
\begin{proof}
We want to use the results from \cite{DMT95}. In order to do so, we use the fact that according to Corollary \ref{C:small} the compact set $ \overline{B(0,R)}\times \mathbb{S}_1$ is small.\\
\\
Furthermore, according to Proposition \ref{p:expintegrability} we know that a function $V_R$ can be defined by 
\begin{equation}\label{e:Lfunction}
V_R(\xi,\alpha)=
1 + \mathbb{E}_{x=(\xi,\alpha)}\biggl[\int_0^{\tau_R(\delta)}e^{\lambda t}\,dt\biggr] 
\end{equation}
where the $\delta,\eta>0$ are some constants and $\tau_R(\delta) = \inf \lbrace t>\delta \mid X_t \in B(0,R)\times \mathbb{S}_1\rbrace$. Observe that at time $\delta>0$ the $\xi$-component of the process travelled at most a distance $\delta$, i.e. if $\xi_0=\xi$ then $|\xi_{\delta}-\xi|\leq \delta$ and therefore existence of an exponential moment for the random variable $\tau_{R}$ implies the existence of an exponential moment of the random variable $\tau_{R}(\delta)$. This shows that $V_R$ is indeed well-defined.  
Then $V_R$ is locally bounded and therefore by Theorem 6.2 in \cite{DMT95} the drift conditions needed in Theorem 5.2 in \cite{DMT95} are satisfied. Therefore we can apply Theorem 5.2 in \cite{DMT95} in order to deduce geometric ergodicity. 
\end{proof}
The proof actually also allows to conclude the following slightly stronger statement. First, let us recall the definition of the weighted supremum spaces $\mathcal{B}_{V_R}(S)$,
\begin{displaymath}
\mathcal{B}_{V_R}(S) := \biggl\lbrace f : S \rightarrow \mathbb{R} \mid \sup_{x \in S}\frac{|f(x)|}{V_R(x)}<\infty\biggr\rbrace 
\end{displaymath}
equipped with the norm $\|\cdot\|_{V_R}=\sup_{x \in S}\frac{|f(x)|}{V_R(x)}$.
\begin{Corollary}
Under the conditions of Theorem \ref {t:main} the transitions semigroup $(P_t)_{t \geq 0}$ admits a spectral gap in the sense that for some constants $C>0$ and $v>0$
\begin{displaymath}
\bigl\| P_t-\mathbf{1}\otimes \mu \bigr\|_{{V_R}}\leq C \, e^{-v t},
\end{displaymath}
where $\mathbf{1}\otimes \mu$ is the projection on the constant functions defined by $\mathbf{1}\otimes \mu(f)=\int f\,d\mu\cdot \mathbf{1}$.
\end{Corollary}
\begin{proof}
This is shown in \cite{DMT95} to be a consequence of Theorem \ref{t:main}.
\end{proof}
In the functional analytic language of \cite{Wu04} and \cite{H} this can be formulated in the following way
\begin{Corollary}
Under the conditions of Theorem \ref{t:main} the operator semigroup $(P_t)_{t \geq 0}$ is quasi--compact in the Banach space $B_{V_R}(S)$.
\end{Corollary}
Thus it remains to prove Proposition \ref{p:expintegrability}. This will be the content of the next subsection. Before we start we have to emphasize once more that in this work we will be interested in the first hitting time of the process $(\beta_t,r_t)_{t\geq 0}$ of sets of the form $\mathbb{S}_1 \times [0,R]$, where $R>0$ is chosen large enough such that $b(r)-1/r\geq c$ for some $c>0$ and all $r\geq R$. Thus we do not need to worry about the $1/r$-singularity of the drift in \eqref{e:radialSDE}. In fact we can without mentioning this every time replace $b(r)-1/r$ by some smooth and bounded function $\tilde{b}$. 
\subsection{Proof of Proposition \ref{p:expintegrability}}
The proof is divided into smaller parts.  

\begin{lemma}\label{l:tanaka}
Let $(r_t,\beta_t)_t$ denote a solution of SDE \eqref{e:radialSDE} starting at $(r_0,\pi)$ then there exists a reflected Brownian motion $(R_t)_{t\geq 0}$ starting at $0$ such that almost surely
\begin{displaymath}
\bigl|\pi - \beta_t\bigr| \leq \sigma R_t 
\end{displaymath} 
for all $t \leq \rho_{(\pi/2,3\pi/2)} =\inf \lbrace t \geq 0 \mid \beta_t\  \notin (\pi/2,3\pi/2)\rbrace$. Moreover, one has almost surely $|\beta_{\rho^R}-\pi|<\pi/2$, where $\rho^R=\inf\lbrace t > 0 | R_t = \pi/2 \rbrace$.
\end{lemma}
\begin{proof}
For the proof we set $\sigma=1$. Observe that due to 
\begin{displaymath}
d\beta_t = dB_t + \bigl(b(r_t)-r_t^{-1}\bigr)\sin(\beta_t)\,dt
\end{displaymath}
the process $(\beta_t)_{t}$ defines a continuous semimartingale. Thus we can use the It\^o - Tanaka formula in order to deduce that
\begin{equation}
\begin{split}
d|\beta_t-\pi| &= \text{sgn}(\beta_t-\pi)\,d\beta_t + dL^{\beta,\pi}_t \\
&=  \text{sgn}(\beta_t-\pi)\,dB_t +  \text{sgn}(\beta_t-\pi)\bigl(b(r_t)-r_t^{-1}\bigr)\sin(\beta_t)\,dt + dL^{\beta,\pi}_t ,
\end{split}
\end{equation}
where $L^{\beta,\pi}$ denotes the local time of $\beta$ in $\pi$.  Observe now that for $t\leq \rho_{(\pi/2,3\pi/2)}$ we have 
\begin{equation}\label{e:local-nonneg}
\text{sgn}(\beta_t-\pi)\bigl(b(r_t)-r_t^{-1}\bigr)\sin(\beta_t) \leq 0
\end{equation}
and that by the Levy characterization the process $\bigl(\int_0^t\text{sgn}(\pi-\beta_s)\,dB_s \bigr)_{t \geq 0}$ is a Brownian motion. Moreover, by Theorem 22.1 in \cite{KalBook} we have that 
\begin{displaymath}
L^{\beta,\pi}_t = -\inf_{s \leq t}\int_0^s \text{sgn}(\beta_u-\pi)\,d\beta_u.
\end{displaymath}

Let us define $t_0$ by
\begin{equation}\label{e:def-t0}
\inf_{0 \leq s \leq t}\int_0^s \text{sgn}(\beta_u-\pi)\,d\beta_u = \int_0^{t_0} \text{sgn}(\beta_u-\pi)\,d\beta_u
\end{equation}
and let $t_1$ satisfy
\begin{equation}\label{e:def-t1}
\inf_{0\leq s \leq t}\int_0^s \text{sgn}(\beta_u-\pi)\,dB_u = \int_0^{t_1} \text{sgn}(\beta_u-\pi)\,dB_u.
\end{equation}
Assume that $t_1>t_0$, then we by the very definition know that
\begin{equation}\label{e:cons}
\begin{split}
 \int_0^{t_0} &\text{sgn}(\beta_u-\pi)\,d\beta_u= \int_0^{t_0} \text{sgn}(\beta_u-\pi)\,dB_u +  \int_0^{t_0} \text{sgn}(\beta_u-\pi)\delta_u\,du \\
 &\leq  \int_0^{t_1} \text{sgn}(\beta_u-\pi)\,d\beta_u= \int_0^{t_1} \text{sgn}(\beta_u-\pi)\,dB_u +  \int_0^{t_1} \text{sgn}(\beta_u-\pi)\delta_u\,du,
\end{split}
\end{equation}
where $\delta_u := \bigl(b(r_t)-r_t^{-1}\bigr)\sin(\beta_t)$. 

We might assume to have a strict inequality, since otherwise we could take $t_0=t_1$. Now a strict inequality in \eqref{e:cons} implies that 
\begin{displaymath}
\int_{t_0}^{t_1}\text{sgn}(\beta_u-\pi)\,dB_u \geq - \int_{t_0}^{t_1}   \text{sgn}(\beta_u-\pi)\delta_u \,dy>0.
\end{displaymath}
But this gives us
\begin{equation*}
\begin{split}
\int_0^{t_1} \text{sgn}(\beta_u-\pi)\,dB_u &= \int_0^{t_0} \text{sgn}(\beta_u-\pi)\,dB_u + \int_{t_0}^{t_1} \text{sgn}(\beta_u-\pi)\,dB_u \\
&\geq \int_0^{t_0} \text{sgn}(\beta_u-\pi)\,dB_u 
\end{split}
\end{equation*}
contradicting our choice of $t_1$ in \eqref{e:def-t1} and the assumption $t_1>t_0$. 
 
In the case $t_1\leq t_0$ we proceed as follows :  We have
\begin{equation*}
\begin{split}
|\beta_t-\pi| &= \int_0^t\text{sgn}(\beta_u-\pi)\,dB_u + \int_0^t \text{sgn}(\beta_u-\pi)\delta_u\,du - \inf_{s\leq t}\int_0^s\text{sgn}(\beta_u-\pi)\,d\beta_u\\
&= \int_0^t\text{sgn}(\beta_u-\pi)\,dB_u + \int_0^t \text{sgn}(\beta_u-\pi)\delta_u\,du -  \int_0^{t_0}\text{sgn}(\beta_u-\pi)\,dB_u \\
&\quad + \int_0^{t_0}\text{sgn}(\beta_u-\pi)\delta_u\,du\\
&\leq  \int_0^t\text{sgn}(\beta_u-\pi)\,dB_u + \int_{t_0}^t \text{sgn}(\beta_u-\pi)\delta_u\,du - \inf_{s\leq t}\int_0^s\text{sgn}(\beta_u-\pi)\,dB_u\\
&\leq  \int_0^t\text{sgn}(\beta_u-\pi)\,dB_u - \int_0^{t_0}\text{sgn}(\beta_u-\pi)\,dB_u,
\end{split}
\end{equation*}
where we used \eqref{e:local-nonneg} in the last step. This gives in the case  $t_1\leq t_0$ 
\begin{equation}
\begin{split}
|\beta_t-\pi| \leq \int_0^t \text{sgn}(\beta_u-\pi)\,dB_u  - \inf_{s\leq t} \int_0^s \text{sgn}(\beta_u-\pi)\,dB_u,
\end{split}
\end{equation}
and since according to Theorem 2.34 in \cite{MPBook} the right hand side is distributed as the reflected Brownian motion we arrive at the assertion of the lemma.
\end{proof}
In a completely analogous way one proves
\begin{lemma}\label{l:tanaka2}
Let $(r_t,\beta_t)_t$ denote a solution of SDE \eqref{e:radialSDE} starting at $(r_0,2\pi)$ then there exists a reflected Brownian motion $(R_t)_{t\geq 0}$ starting at $0$ such that almost surely
\begin{displaymath}
\bigl|2\pi - \beta_t\bigr| \geq \sigma R_t 
\end{displaymath} 
for all $t \leq \rho_{(\pi/2,3\pi/2)} =\inf \lbrace t \geq 0 \mid \beta_t\  \notin (\pi/2,3\pi/2)\rbrace$. Moreover, one has almost surely $|\beta_{\rho^R \wedge \rho_{(\pi/2,3\pi/2)}}-\pi|>3\pi/2$, where $\rho^R =\inf\lbrace t > 0 | R_t = \pi/2 \rbrace$.

\end{lemma}

\subsection{Special Case}
Assume that $b(r)=c+1/r$ for $r\geq R$. Then the process $(\beta_t)_{t \geq 0}$ satisfies $d\beta_t = \sigma dB_t + c\sin(\beta_t)\,dt$ up to the random time
\begin{displaymath}
\tau_R = \inf \biggl\lbrace t \geq 0\mid \int_0^t\cos(\beta_t)\,dt \leq R\biggr\rbrace.
\end{displaymath}
This is a very special case, but it already contains all features necessary to understand the general situation.  Moreover, we make use of the process $(\gamma_t)_{t \geq 0}$, which is defined in the following way. 
\begin{itemize}
\item[i)] $\gamma_0=\pi$; $\gamma_t$ behaves like a Brownian motion reflected at $\pi$ up to the first hitting time of $3\pi/2$. 
\item[ii)] Then $\gamma_t$ behaves like the solution of  $d\beta_t = \sigma dB_t + c\,\sin(\beta_t)\,dt$ until it hits $\pi$ or $2\pi$. 
\item[iii)] If in step ii) $\gamma_t$ hit $\pi$, then we proceed as in item i). If in step ii) $\gamma_t$ hit $2\pi$, then we the process $\gamma_t$ is continued as a Brownian motion reflected at $2\pi$ up to the first hitting time of $3\pi/2$. After this time we proceed as in item ii).
\end{itemize}
Let us denote by $(\tilde{\beta}_t)_{t\geq 0}$ the solution of $d\tilde{\beta}_t = \sigma dB_t + c\,\sin(\tilde{\beta}_t)\,dt$ started at $\pi$ and we set $\tilde{r}_t= \int_0^t\cos(\tilde{\beta}_s)\,ds$. The process $(\tilde{\beta}_t)_{t\geq 0}$ is still considered as a process on the circle, i.e. we identify $0$ and $2\pi$. Thus up to the first hitting time $\tau_R=\tilde{\tau}_R$ of the set $\mathbb{S}_1 \times [0,R]$ the two processes $(\beta_t,r_t)_{t \geq 0}$ and $(\tilde{\beta}_t,\tilde{r}_t)_{t\geq 0}$ coincide if $\beta_0=\tilde{\beta}_0$ and $r_0 = \tilde{r}_0$. We now define recursively the stopping times
\begin{displaymath}
\sigma_0=0=\sigma^0, \sigma_1=\inf\lbrace t \geq 0 \mid \tilde{\beta}_t \notin (0,2\pi)\rbrace, \tilde{\sigma}^1=\inf\lbrace t \geq 0 \mid \tilde{\beta}_t= \pi \rbrace, \sigma^1=\sigma_1+\tilde{\sigma}^1 \circ\theta_{\sigma_1}
\end{displaymath}
and 
\begin{displaymath}
\sigma_n = \sigma^{n-1} + \sigma_{1} \circ \theta_{\sigma^{n-1}},\, \sigma^n=\sigma_{n} + \tilde{\sigma}^{1} \circ \theta_{\sigma_{n}}.
\end{displaymath}
The analog stopping times are defined for the process $(\gamma)_{t \geq 0}$ and by a small abuse of notation we denote these also by $(\sigma^i)_{ i \in \mathbb{N}_0}$ and $(\sigma_i)_{i \in \mathbb{N}_0}$. 

We say that the process $(\beta_t)_{t \geq 0}$ (resp. $(\gamma_t)_{t\geq 0}$) started at $\pi$ completes a cycle if it returns to $\pi$ after having visited $2\pi$. Thus the stopping time $\sigma^i$ describe the time of completion of the $i$-{th} cycle. 
 
Moreover, for some interval we denote by $\tilde{\lambda}^{I}_0$ the smallest Dirichlet-eigenvalue of the operator $-\frac{1}{2}\frac{d^2}{dx^2}-c\sin(x)\frac{d}{dx}$, which can also be characterized by
\begin{displaymath}
\lambda_0^I=-\lim_{t \rightarrow \infty}\frac{1}{t}\log\sup_{z \in I}\mathbb{P}_z\bigl(\rho_I>t\bigr),
\end{displaymath} 
where $\rho_I=\inf\lbrace t \geq 0 \mid \tilde{\beta}_t \notin I\rbrace$ denotes the first exit time. We set
\begin{displaymath}
\tilde{\lambda}_0=\min\bigl(\lambda_0^{(0,\pi)},\lambda_0^{(\pi/2,3\pi/2)},\lambda_0^{(\pi,2\pi)}\bigr).
\end{displaymath}
\begin{lemma}
The random variables
\begin{displaymath}
X_i=\int_{\sigma^{i-1}}^{\sigma^{i}}\cos(\tilde{\beta}_s)\,ds, \, i=1,2,\dots
\end{displaymath}
form an i.i.d sequence of random variables having the following properties:
\begin{itemize}
\item The random variables $X_i$ ($i=1,2,\dots$) have a negative expectation, i.e.
\begin{displaymath}
\mathbb{E}_{\tilde{\beta}_0=\pi}\bigl[X_i\bigr]<0.
\end{displaymath}
\item For all $\lambda < \tilde{\lambda}_0$ we have
\begin{displaymath}
\mathbb{E}_{\tilde{\beta}_0=\pi}\bigl[e^{\lambda|X_i|}\bigr]<\infty.
\end{displaymath}
\end{itemize}
The same properties hold true for the sequence $(X^{\gamma}_{i})_{ i \in \mathbb{N}}$ where $X^\gamma_i=\int_{\sigma^{i-1}}^{\sigma^{i}}\cos(\gamma_s)\,ds, \, i=1,2,\dots$
\end{lemma}
\begin{remark}
Though we are not going to present a proof here it seems possible to show, that 
\begin{displaymath}
\lim_{\sigma \rightarrow \infty} \frac{\mathbb{E}_{\tilde{\beta}_0=\pi}\bigl[X_1\bigr]}{\mathbb{E}_{\tilde{\beta}_0=\pi}[\sigma^1]} = 0.
\end{displaymath}
Heuristically this is to be expected since for a large diffusion constant the influence of the drift becomes eventually negligible. Such an assertion might be finally used in order to argue that enlarging the diffusion constant $\sigma$ does not lead to faster rate of convergence. Instead one even has to expect that the rate of convergence will become slower. This is due to the fact that the value $\mathbb{E}_{\tilde{\beta}_0=\pi}\bigl[X_i\bigr]$ controls in a certain sense the 'speed' with which the process $(X_t)_{t\geq 0}$ returns to a ball around the origin. Bigger values of $\mathbb{E}_{\tilde{\beta}_0=\pi}\bigl[X_i\bigr]$ heuristically lead to longer return times to the ball and thus one expects to get a smaller spectral gap. 
\end{remark}
\begin{proof}
The fact that the sequence $(X_i)_{i \in \mathbb{N}}$ is  i.i.d is a direct consequence of the strong Markov property of the process $(\tilde{\beta}_t)_{t \geq 0}$. 

The first assertion can be shown in the following way. We first observe that 
\begin{equation}\label{e:comparisoneasy}
\begin{split}
\mathbb{E}_{\tilde{\beta}_0=\pi}\bigl[X_i\bigr] &= \mathbb{E}_{\tilde{\beta}_0=\pi}\biggl[\int_{0}^{\sigma^1}\cos(\tilde{\beta}_s)\,ds,\biggr]\\
&=\mathbb{E}_{\tilde{\beta}_0=\pi}\biggl[\int_{0}^{\sigma_1}\cos(\tilde{\beta}_s)\,ds,\biggr] + \mathbb{E}_{\tilde{\beta}_0=2\pi}\biggl[\int_{0}^{\tilde{\sigma}^1}\cos(\tilde{\beta}_s)\,ds,\biggr].
\end{split}
\end{equation}
We will now compare the first and the second summand in equation \eqref{e:comparisoneasy}. By the strong Markov property we have
\begin{equation*}
\begin{split}
\mathbb{E}_{\tilde{\beta}_0=2\pi}\biggl[\int_{0}^{\tilde{\sigma}^1}\cos(\tilde{\beta}_s)\,ds\biggr] &= \mathbb{E}_{\tilde{\beta}_0=2\pi}\biggl[\int_{0}^{\rho_{(3\pi/2,5\pi/2)}}\cos(\tilde{\beta}_s)\,ds\biggr] + \mathbb{E}_{\tilde{\beta}_0=2\pi}\biggl[\int_{\rho_{(3\pi/2,5\pi/2)}}^{\tilde{\sigma}_1}\cos(\tilde{\beta}_s)\,ds\biggr] \\
&= \mathbb{E}_{\tilde{\beta}_0=2\pi}\biggl[\int_{0}^{\rho_{(3\pi/2,5\pi/2)}}\cos(\tilde{\beta}_s)\,ds\biggr] \\
&+\frac{1}{2}\mathbb{E}_{\tilde{\beta}_0=3\pi/2}\biggl[\int_0^{\tilde{\sigma}_1}\cos (\tilde{\beta}_s)\,ds\biggr] + \frac{1}{2}\mathbb{E}_{\tilde{\beta}_0=\pi/2}\biggl[\int_0^{\tilde{\sigma}_1}\cos (\tilde{\beta}_s)\,ds\biggr] 
\end{split}
\end{equation*}
and then
\begin{equation}
\begin{split}
\mathbb{E}_{\tilde{\beta}_0=2\pi}\biggl[\int_{0}^{\tilde{\sigma}^1}\cos(\tilde{\beta}_s)\,ds\biggr] &= \mathbb{E}_{\tilde{\beta}_0=2\pi}\biggl[\int_{0}^{\rho_{(3\pi/2, 5\pi/2)}}\cos(\tilde{\beta}_s)\,ds\biggr]  \\
&+ \frac{1}{2}\mathbb{E}_{\tilde{\beta}_0=3\pi/2}\biggl[\int_0^{\rho_{(\pi,2\pi)}}\cos (\tilde{\beta}_s)\,ds\bigr] \\
&+ \frac{1}{2}\mathbb{P}_{\tilde{\beta}_0=3/2\pi}\bigl(\tilde{\sigma}^1 > \rho_{(\pi,2\pi)}\bigr)\mathbb{E}_{2\pi}\biggl[\int_0^{\tilde{\sigma}^1}\cos (\tilde{\beta}_s)\,ds\biggr]  \\
&+\frac{1}{2}\mathbb{E}_{\tilde{\beta}_0=\pi/2}\biggl[\int_0^{\rho_{(0,\pi)}}\cos (\tilde{\beta}_s)\,ds\biggr] \\
&+  \frac{1}{2}\mathbb{P}_{\tilde{\beta}_0=\pi/2}\bigl(\tilde{\sigma}^1 > \rho_{(0,\pi)}\bigr)\mathbb{E}_{2\pi}\biggl[\int_0^{\tilde{\sigma}^1}\cos (\tilde{\beta}_s)\,ds\biggr]  
\end{split}
\end{equation}
where $\rho_{I}$ denotes the first exit time from the interval $I$. Thus with $d_1=1-\mathbb{P}_{\tilde{\beta}_0=3\pi /2}\bigl(\tilde{\sigma}^1 > \rho_{(0,\pi)}\bigr)$ we get using obvious symmetry properties
\begin{equation}\label{e:negmeanI}
\begin{split}
\mathbb{E}_{\tilde{\beta}_0=2\pi}\biggl[\int_{0}^{\tilde{\sigma}^1}\cos(\tilde{\beta}_s)\,ds\biggr] &= d_1^{-1}\biggl(\mathbb{E}_{\tilde{\beta}_0=2\pi}\biggl[\int_{0}^{\rho_{(3\pi/2, 5\pi/2)}}\cos(\tilde{\beta}_s)\,ds\biggr]  \\
&+ \mathbb{E}_{\tilde{\beta}_0=3\pi/2}\biggl[\int_0^{\rho_{(\pi,2\pi)}}\cos (\tilde{\beta}_s)\,ds\biggr]\biggr).
\end{split}
\end{equation}

Analogously we get
\begin{equation}
\begin{split}
\mathbb{E}_{\tilde{\beta}_0=\pi}&\biggl[\int_{0}^{\sigma_1}\cos(\tilde{\beta}_s)\,ds,\biggr] = \mathbb{E}_{\tilde{\beta}_0=\pi}\biggl[\int_{0}^{\rho_{(\pi/2,3\pi/2)}}\cos(\tilde{\beta}_s)\,ds\biggr] + \mathbb{E}_{\tilde{\beta}_0=\pi}\biggl[\int_{\rho_{(\pi/2,3\pi/2)}}^{\sigma_1}\cos(\tilde{\beta}_s)\,ds\biggr] \\
&= \mathbb{E}_{\tilde{\beta}_0=\pi}\biggl[\int_{0}^{\rho_{(\pi/2,3\pi/2)}}\cos(\tilde{\beta}_s)\,ds\biggr] \\
&+\frac{1}{2}\mathbb{E}_{\tilde{\beta}_0=\pi/2}\biggl[\int_0^{\sigma_1}\cos (\tilde{\beta}_s)\,ds\biggr] + \frac{1}{2}\mathbb{E}_{3\pi/2}\biggl[\int_0^{\sigma_1}\cos (\tilde{\beta}_s)\,ds\biggr] \\
&= \mathbb{E}_{\tilde{\beta}_0=\pi}\biggl[\int_{0}^{\rho_{(\pi/2,3\pi/2)}}\cos(\tilde{\beta}_s)\,ds\biggr]  \\
&+ \frac{1}{2}\mathbb{E}_{\tilde{\beta}_0=\pi/2}\biggl[\int_0^{\rho_{(0,\pi)}}\cos (\tilde{\beta}_s)\,ds\biggr] + \frac{1}{2}\mathbb{P}_{\tilde{\beta}_0=\pi/2}(\sigma_1>\rho_{(0,\pi)})\mathbb{E}_{\pi}\biggl[\int_0^{\sigma_1}\cos (\tilde{\beta}_s)\,ds\biggr]  \\
&+\frac{1}{2}\mathbb{E}_{\tilde{\beta}_0=3\pi/2}\biggl[\int_0^{\rho_{(\pi,2\pi)}}\cos (\tilde{\beta}_s)\,ds\biggr] + \frac{1}{2}\mathbb{P}_{\tilde{\beta}_0=3\pi/2}(\sigma_1>\rho_{(\pi,2\pi)})\mathbb{E}_{\tilde{\beta_0}=\pi}\biggl[\int_0^{\sigma_1}\cos (\tilde{\beta}_s)\,ds\biggr]  
\end{split}
\end{equation}
and with $d_2=1-\mathbb{P}_{\tilde{\beta}_0=3\pi/2}(\sigma_1>\rho_{(\pi,2\pi)})$
\begin{equation}\label{e:negmeanII}
\begin{split}
\mathbb{E}_{\tilde{\beta}_0=\pi}\biggl[\int_{0}^{\sigma_1}\cos(\tilde{\beta}_s)\,ds,\biggr] &= d_2^{-1}\biggl( \mathbb{E}_{\tilde{\beta}_0=\pi}\biggl[\int_{0}^{\rho_{(\pi/2,3\pi/2)}}\cos(\tilde{\beta}_s)\,ds\biggr]  \\
&+\mathbb{E}_{\tilde{\beta}_0=3\pi/2}\biggl[\int_0^{\rho_{(\pi,2\pi)}}\cos (\tilde{\beta}_s)\,ds\biggr] \biggr) 
\end{split}
\end{equation}
Using Lemma \ref{l:tanaka}  together with the properties of the cosine we conclude that 
\begin{equation}\label{e:compinextrema}
-\mathbb{E}_{\tilde{\beta}_0=\pi}\biggl[\int_{0}^{\rho_{(\pi/2,3\pi/2)}}\cos(\tilde{\beta}_s)\,ds\biggr] \geq -\mathbb{E}_{\tilde{\beta}_0=\pi}\biggl[\int_{0}^{\rho_{(\pi/2,3\pi/2)}}\cos(B_s)\,ds\biggr],
\end{equation}
where $(B_s)$ denotes a Brownian motion in $[0,2\pi]$ with periodic boundary conditions and that 
\begin{displaymath}
\mathbb{E}_{\tilde{\beta}_0=2\pi}\biggl[\int_{0}^{\rho_{(3\pi/2,5\pi/2)}}\cos(\tilde{\beta}_s)\,ds\biggr] \leq  \mathbb{E}_{\tilde{\beta}_0=2\pi}\biggl[\int_{0}^{\rho_{(\pi/2,3\pi/2)}}\cos(B_s)\,ds\biggr].
\end{displaymath}
Therefore
\begin{displaymath}
\mathbb{E}_{\tilde{\beta}_0=\pi}\biggl[\int_{0}^{\rho_{(\pi/2,3\pi/2)}}\cos(\tilde{\beta}_s)\,ds \biggr] \leq \mathbb{E}_{\tilde{\beta}_0=2\pi}\biggl[\int_{0}^{\rho_{(3\pi/2,5\pi/2)}}\cos(\tilde{\beta}_s)\,ds\biggr] .
\end{displaymath}
Using Theorem 3.10 in the Appendix or alternatively explicit expression for hitting probabilities from Theorem 23.7 in \cite{KalBook}  it is straightforward to deduce that 
\begin{displaymath}
\mathbb{P}_{\tilde{\beta}_0 =3\pi/2}\bigl(\tilde{\sigma}^1>\rho_{(\pi,2\pi)}\bigr)<\mathbb{P}_{\tilde{\beta}_0 =3\pi/2}\bigl(\sigma_1>\rho_{(\pi,2\pi)}\bigr)
\end{displaymath}
and therefore $d^{-1}_1 < d^{-1}_2$.

Now comparison of \eqref{e:negmeanI} and \eqref{e:negmeanII} using \eqref{e:compinextrema} as well as the following Lemma \ref{l:hilfslemma} together with $d_1^{-1} \leq d_2^{-1}$ allows to deduce the first assertion. 
\\
\\
In order to prove the second assertion observe that $|X_1| \leq \sigma^1$. Thus we get using the strong Markov property again
\begin{equation}
\begin{split}
\mathbb{E}_{\tilde{\beta}_0=\pi}\bigl[ e^{\lambda |X_1|}\bigr]&\leq \mathbb{E}_{\tilde{\beta}_0=\pi}\bigl[e^{\lambda \sigma^1}\bigr]=\mathbb{E}_{\tilde{\beta}_0=\pi}\bigl[ e^{\lambda \sigma_1}\bigr] \mathbb{E}_{\tilde{\beta}_0=2\pi}[ e^{\lambda \tilde{\sigma}^1}\bigl].
\end{split}
\end{equation}
According to our choice of $\tilde{\lambda}_0$ the right hand side remains finite for $\lambda < \tilde{\lambda}_0$. 

The assertions for $(X_i^{\gamma})_{i=1}^{\infty}$ can be proved in an analogous way.  The first assertion remains true without any change. For the second assertion observe that $X_1^{\gamma}$ has exponential moments up to $\tilde{\lambda}_0$ since for every $I=(0,\pi), (\pi/2,3\pi/2),(\pi,2\pi)$ it follows in a straightforward way from Theorem \ref{t:IW} --recalled in the Appendix -- that 
\begin{displaymath}
-\lim_{t \rightarrow \infty}\frac{1}{t}\log\sup_{x \in I}\mathbb{P}_x\bigl(\lbrace \gamma_s \in I \,\, \forall \,0\leq s\leq t\rbrace \bigr)\geq \tilde{\lambda}_0
\end{displaymath}
giving the required estimates.
\end{proof}
\begin{lemma} \label{l:hilfslemma}
We have
\begin{itemize}
\item[a)]
\begin{displaymath}
\mathbb{E}_{3\pi/2}\biggl[\int_0^{\rho_{(\pi,2\pi)}}\cos (\tilde{\beta}_s)\,ds\biggr]<0.
\end{displaymath}
\item[b)] 
\begin{displaymath}
\mathbb{E}_{\pi/2}\biggl[\int_0^{\rho_{(0,\pi)}}\cos (\tilde{\beta}_s)\,ds\biggr]<0.
\end{displaymath}
\end{itemize}
Both assertions remain obviously true, if $(\tilde{\beta}_t)_{t \geq 0}$ is replaced by $(\gamma_t)_{t \geq 0}$.
\end{lemma}
\begin{proof}
We only prove assertion a) and in order to simplify the formulas we set $\sigma = 1$ and $c=1$. In order to prove the general case only notational changes are necessary. We have 
\begin{equation}
\begin{split}
\mathbb{E}_{3\pi/2}\biggl[\int_0^{\rho_{(\pi,2\pi)}}\cos (\tilde{\beta}_s)\,ds\biggr]&=\mathbb{E}_{3\pi/2}\biggl[\int_0^{\infty}\cos (\tilde{\beta}_s) \mathbf{1}_{ \lbrace \rho_{(\pi,2\pi)}>s\rbrace}\,ds\biggr]\\
&=\int_0^{\infty}\mathbb{E}_{3\pi/2}\bigl[\cos (\tilde{\beta}_s) \mathbf{1}_{ \lbrace \rho_{(\pi,2\pi)}>s\rbrace}\bigr]\,ds.
\end{split}
\end{equation}
Thus it is sufficient to show that $\mathbb{E}_{3\pi/2}\bigl[\cos (\tilde{\beta}_t) \mathbf{1}_{ \lbrace \rho_{(\pi,2\pi)}>t\rbrace}\bigr]<0$ for $t>0$. By Girsanov's theorem we have
\begin{equation}
\begin{split}
\mathbb{E}_{3\pi/2}\bigl[\cos (\tilde{\beta}_t) &\mathbf{1}_{ \lbrace \rho_{(\pi,2\pi)}>t\rbrace}\bigr]=\mathbb{E}_{3\pi/2}\bigl[\cos (B_t) e^{\int_0^t\sin(B_s)\,dB_s-\frac{1}{2}\int_0^t\sin^2(B_s)\,ds}\mathbf{1}_{ \lbrace \rho_{(\pi,2\pi)}>t\rbrace}\bigr] \\
&= e^{\cos(3\pi/2)}\mathbb{E}_{3\pi/2}\bigl[\cos (B_t) e^{-\cos(B_t)}e^{-\frac{1}{2}\int_0^t\cos(B_s)\,ds-\frac{1}{2}\int_0^t\sin^2(B_s)\,ds}\mathbf{1}_{ \lbrace \rho_{(\pi,2\pi)}>t\rbrace}\bigr],
\end{split}
\end{equation}
where we used, that according to the It\^o formula
\begin{displaymath}
-\cos(B_t)=-\cos(B_0)+\int_0^t\sin(B_s)\,dB_s + \frac{1}{2}\int_0^t\cos(B_s)\,ds.
\end{displaymath}
Finally by considering the excursion straddling $t$ and observing that the cosine is negative in $[\pi,3\pi/2)$ and positive in $[3\pi/2,2\pi)$ we conclude that 
\begin{equation}\label{e:symmetryargument}
\begin{split}
\mathbb{E}_{3\pi/2}\bigl[\cos (B_t) &e^{-\cos(B_t)}e^{-\frac{1}{2}\int_0^t\cos(B_s)\,ds-\frac{1}{2}\int_0^t\sin^2(B_s)\,ds}\mathbf{1}_{ \lbrace \rho_{(\pi,2\pi)}>t\rbrace}\bigr]\\
&= \mathbb{E}_{3\pi/2}\bigl[\cos (B_t) e^{-\cos(B_t)}e^{-\frac{1}{2}\int_0^t\cos(B_s)\,ds-\frac{1}{2}\int_0^t\sin^2(B_s)\,ds}\mathbf{1}_{ \lbrace \rho_{(\pi,2\pi)}>t\rbrace};B_t<3\pi/2\bigr]\\
&+\mathbb{E}_{3\pi/2}\bigl[\cos (B_t) e^{-\cos(B_t)}e^{-\frac{1}{2}\int_0^t\cos(B_s)\,ds-\frac{1}{2}\int_0^t\sin^2(B_s)\,ds}\mathbf{1}_{ \lbrace \rho_{(\pi,2\pi)}>t\rbrace};B_t>3\pi/2\bigr]\\
&= I_1 + I_2 <0.
\end{split}
\end{equation}
In order to see this note first that we are using the symmetry of the expression $\int_0^t \sin^2(B_s)\,ds$ in order to restrict ourselves to the expressions containing the cosine.  Then one should observe that the contribution of the path $(B_s)_{s \leq l(t)}$ with $l(t)= \sup \lbrace u \leq t \mid B_u = 3\pi/2\rbrace$ to the integrands in $I_1$ and $I_2$ in \eqref{e:symmetryargument} is the same and that therefore only the contribution of $(B_s)_{l(t)\leq s \leq t}$ makes the difference. This completes the proof of the assertions.
\end{proof}
Define $(S_n)_{n \geq 1}$ to be the random walk corresponding to the i.i.d sequence $(X_i)$, i.e. $S_n=\sum_{i=1}^nX_i$. Moreover let $T_-$ denote the time
\begin{displaymath}
T_-=\inf\bigl\lbrace n \geq 0 \mid S_n \leq R \bigr\rbrace.
\end{displaymath}
Similiarly, we define $(S_n)_{n\geq 1}$ to be the random walk defined by $S^{\gamma}_n=\sum_{i=1}^{n}X_i^{\gamma}$ and let $T^{\gamma}_-$ be the first hitting time of $[-\infty,R]$. 
We first note that our results imply the existence of exponential moments for $T_-$ and $T^{\gamma}_-$, respectively.
\begin{lemma}\label{l:exponRW}
Let $\nu$ denote an initial distribution on $[R,\infty]$ having exponentially decaying tails, i.e. there exists $c_0>0$ and $c_1 \in (0,1)$ such that $\nu((t,\infty)) \leq c_0\,c_1^t$. Then there exists a strictly positive $\lambda'$ depending only on $c_1$ such that 
\begin{displaymath}
\mathbb{E}_{\nu}\bigl[e^{\lambda' T_-}\bigr]<\infty\,\text{   and   }\,\mathbb{E}_{\nu}\bigl[e^{\lambda' T^{\gamma}_-}\bigr]<\infty
\end{displaymath}
\end{lemma}
\begin{proof}
By shifting the problem using translation invariance of the random walk we can assume that $R=0$. First we note that according to \cite{Hey} we have for every initial point $s_0>R=0$ 
\begin{displaymath}
\mathbb{E}_{s_0}\bigl[e^{\bar{\lambda} T_-}\bigr]<\infty
\end{displaymath}
for some $\bar{\lambda} > 0$. Let us fix such an initial point $s_0$ and $\bar{\lambda} \geq \lambda'>0$. Since the measure $\nu$ is supposed to have exponentially decaying tails we have that 
\begin{displaymath}
\forall k \geq 0:\nu((ks_0,(k+1)s_0))\leq c_0c_1^k.
\end{displaymath}  
Define the times $T_{ks_0}$ to be 
\begin{displaymath}
T_{ks_0}=\inf\bigl\lbrace n \geq 0 \mid  S_n\leq ks_0\bigr\rbrace
\end{displaymath}
and observe that when started at $(k+1)s_0$ the hitting time of the negative real numbers can be bounded by first waiting for $T_{ks_0}$, then starting at $ks_0$ we wait for $T_{(k-1)s_0}$ and so on. This gives
\begin{displaymath} 
\mathbb{E}_{(k+1)s_0}\bigl[e^{\lambda' T_-}\bigr] \leq \mathbb{E}_{s_0}\bigl[e^{\lambda' T_{-}}\bigr]^{k+1} 
\end{displaymath}
Therefore
\begin{displaymath}
\mathbb{E}_{\nu}\bigl[ e^{\lambda T_-}\bigr] \leq c_0 \mathbb{E}_{s_0}\bigl[e^{\lambda' T_{-}}\bigr]\,\sum_{k=0}^nc_1^k\,\mathbb{E}_{s_0}\bigl[e^{\lambda' T_{-}}\bigr]^{k} < \infty,
\end{displaymath}
if $\lambda' >0$ is taken to be small enough. Observe that the choice of $\lambda'$ depends only on $c_1$. The same arguments obviously show the assertion concerning $T^{\gamma}_-$. 
\end{proof}
We denote the stochastic process $\biggl(\int_0^t\cos(\tilde{\beta}_s)\,ds\biggr)_{t\geq 0}$ by $(\tilde{r}_t)$ and we set
\begin{displaymath}
\tilde{\tau}_R = \inf\bigl\lbrace t \geq 0 \mid \tilde{r}_t \leq R\bigr\rbrace.
\end{displaymath}
Moreover, $r^{\gamma}_t$ denotes $\int_0^t\cos(\gamma_s)\,ds$ and $\tau_R^{\gamma}$ is defined to be the first hitting time
\begin{displaymath}
{\tau}^{\gamma}_R = \inf\bigl\lbrace t \geq 0 \mid {r}^{\gamma}_t \leq R\bigr\rbrace.
\end{displaymath}
The assertion of Proposition \ref{p:expintegrability} for the simple case thus follows from the next Lemma.
\begin{lemma}\label{l:exponsimple}
Let $\nu$ denote some initial distribution on $[R,\infty)$ satisfying $\nu((t,\infty)) \leq c_0\,c_1^t$ for some $c_0>0$ and $c_1 \in (0,1)$. Then for some $\lambda > 0$ depending only on $c_1$ and we have
\begin{displaymath}
\mathbb{E}_{\nu}\bigl[e^{\lambda\tilde{ \tau}_R}\bigr] < \infty\,\text{   and   }\,\mathbb{E}_{\nu}\bigl[e^{\lambda \tau^{\gamma}_R}\bigr] < \infty
\end{displaymath}
\end{lemma}
\begin{proof}
According to the construction of the random walk $(S_n)_{n \in \mathbb{N}}$ we conclude that 
\begin{displaymath}
T_{[-\infty,R]} = n \,\text{   implies   }\,\tau_R \text{   happened before completing cycle  } n.
\end{displaymath}
whenever the random walk $S_n$ and the process $r_t$ are started from the same initial distribution. This means that 
\begin{equation}\label{e:expontau}
\mathbb{E}_{\nu}\bigl[e^{\lambda \tau_R}\bigr]  = \sum_{n=1}^{\infty}\mathbb{P}_{\nu}\bigl(T_- = n\bigr)\mathbb{E}_{\nu}\bigl[e^{\lambda \tau_R}\bigr] ^n.
\end{equation}
Using the fact that according to Lemma \ref{l:exponRW} the sequence $(\mathbb{P}_{\nu}\bigl(T_- = n\bigr))_{n \in \mathbb{N}}$ decays exponentially we can choose $\lambda>0$ such that the right hand side of \eqref{e:expontau} remains finite. The same proof applies to the second assertion. 
\end{proof}
\begin{Corollary}
Assume that outside a compact set $\varphi$ is radial and that outside this compact set $b(r)=c+1/r$. Then the assertion of Theorem $\ref{t:main}$ holds true.
\end{Corollary}
\begin{proof}
It remains to show that suitable sufficiently large $R>0$ and every $(r_0,\beta_0) \in (0,\infty) \times \mathbb{S}_1$ we have $\mathbb{E}_{(r_0,\beta_0)}\bigl[e^{\tau_R}\bigr]<\infty$. Observe that by the strong Markov property we have
\begin{displaymath}
\mathbb{E}_{(r_0,\beta_0)}\bigl[e^{\lambda \tau_R}\bigr] \leq \mathbb{E}_{(\nu_{\beta_0,r_0},\pi)}\bigl[e^{\lambda \tau_R}\bigr],
\end{displaymath}
where $\nu_{\beta_0,r_0}$ is some distribution on $(R,\infty)$. Moreover, as the tails of $\nu_{\beta_0,r_0}$ can be dominated by the tails of the distribution of the first hitting time of $\pi$ by the process $(\beta_t)_{t \geq 0}$ we conclude that there exists a constant $c_0=c_{0,\beta_0,r_0}>0$ (possibly depending on $\beta_0$ and $r_0$) and a constant $c_1 \in (0,1)$, which does not depend on $\beta_0$ and $r_0$ such that for $l\geq R$
\begin{displaymath}
 \nu_{\beta_0,r_0}\bigl((l,\infty)) \leq c_0 \, c_1^l.
\end{displaymath}
Therefore there exists some $\lambda>0$ such that for every pair $(\beta_0,r_0)$ we have $\mathbb{E}_{(r_0,\beta_0)}\bigl[e^{\lambda \tau_R}\bigr] < \infty$, showing that Proposition \ref{p:expintegrability} is true in this special situation and thus Theorem \ref{t:main} as well. 
\end{proof}
\begin{remark}
The analysis of this special case demonstrates the dynamical properties of the fiber lay down process in a rather clean way. In order to extract the net-'drift' towards $0$ of the radial component one has to wait until a full cycle is completed. Then each completed cycle can be used as an i.i.d. increment in a random walk which has then negative drift. The identification of the possible translation of the diffusion hitting time problem to a random walk problem constitutes the essential insight.  
\end{remark}
\subsection{General case}
Recall that by definition the process $(\gamma_t)_t$ starts from $\pi$ as a reflected Brownian motion until it reaches $\pi/2$ or $3\pi/2$, then it behaves as the solution of $dX_t=dB_t +c\sin(X_s)\,ds =: dB_t +\delta(X_s)\,ds$ until it hits $\pi$ or $0=2\pi$, from there it starts as a reflected Brownian motion at $2\pi$ until it hits $\pi$ and repeats this behavior. We have seen, that with $r^{\gamma}_t=\int_0^t\cos(\gamma_s)\,ds$ and 
\begin{displaymath}
\tau_R^{\gamma}=\inf \lbrace t > 0 \mid r^{\gamma}_t \leq R \rbrace
\end{displaymath}
we have 
\begin{displaymath}
-\lim_{t \rightarrow \infty}\frac{1}{t}\log \mathbb{P}_{\nu}\bigl(\tau^{\gamma}_R>t \bigr)>0
\end{displaymath}
for every initial distribution $\nu$ on $(0,\infty)$ having exponential moments. The main goal in this subsection consists in proving that it is possible to compare the general case with the process $(\gamma_t)$.
\subsubsection{A Coupling Construction} We are now going to construct a process $(\gamma_t)_{t \geq 0}$ in a way, which allows a comparison to the process $(\beta_t)_{t \geq 0}$. More precisely, $(\gamma_t)_t$ is constructed in such a way that 
\begin{displaymath}
r^{\gamma}_t=\int_0^t\cos(\gamma_s)\,ds \geq r_t.
\end{displaymath}
We want to construct the coupling for one complete cycle of $(\gamma_t)_t$, i.e. we want to construct such a coupling up to the time $\sigma^1$.  The idea is easy to grasp, but as often with coupling constructions, the formal presentation tends to hide the simple underlying intuition.  The idea consists essentially in a comparison of the process $\beta$ with a process which has a weaker drift to $\pi$ and which thus will have a bigger distance to $\pi$. Then one can conclude using the properties of the cosine that the radial process $r$ will be smaller than the corresponding 'radial component' of the comparison process. 

The idea is the following: We use one of the following rules.
\begin{itemize}
\item If the processes $(\gamma_t)_{t \geq 0}$ and $(\beta_t)_{t\geq 0}$ are at a given time both at the point $\pi$ or at $2\pi$ then we use Lemma \ref{l:tanaka} and Lemma \ref{l:tanaka2}, respectively in order to make sure that the distance of the $\gamma$-process to $\pi$ dominated the distance of the $\beta$-process to $\pi$.
\item  If the $\beta$-process is at the point $\pi$ and the $\gamma$-process is in $(\pi,2\pi]$ we wait until the distance of the $\gamma$-process to $\pi$ coincides with the distance of the $\beta$-process to $\pi$. At this time either both processes are in $\pi$ or one is in the situation of the following item.  The same procedure is applied in the case, where the $\gamma$-process is at $2\pi$ and the distance of the $\beta$-process to $\pi$ is strictly smaller than $\pi$.
\item If at a time $\tau$ the $\gamma$-process is in the interior of $(\pi, 2\pi)$ and the $\beta$-process is in $(0,\pi)\cup(\pi,2\pi)$ with $|\beta_{\tau}-\pi| \leq |\gamma_{\tau}-\pi|$ then one can use Theorem \ref{t:IW} in order to conclude that distance of the $\gamma$-process to $\pi$ remains dominated the distance of the $\beta$-process to $\pi$, at least until one hits the boundary and repeats the procedure outlined in the second item. If the processes $\gamma$ and $\beta$ are both in $(\pi,2\pi)$ then the application of Theorem \ref{t:IW} is direct, if the process $\beta$ is in $(0,\pi)$ and the process $\gamma$ is in $(\pi,2\pi)$ then one can use the symmetry properties of the drift of the $\gamma$-process in order to conclude necessary domination of distance of the $\beta$-process to $\pi$ by the distance of the $\gamma$-process to $\pi$ via Theorem \ref{t:IW}.
\end{itemize}
Let us give now the first steps in the construction on a more formal level.
Let $\nu$ be an initial distribution on $[R,\infty)$ having exponentially decaying tails and let $\xi \sim \nu$ be a random variable which is distributed according to $\nu$.  
\begin{itemize}
\item[a)] We start our process $(\beta_t,r_t)_t$ with initial conditions $r_0 = \xi$ and $\beta_0=\pi$. According to Lemma \ref{l:tanaka} there is a Brownian motion $(R_t)_{t\geq 0}$ started in $\pi$ and reflected at $\pi$ such that the distance of $\beta_t$ to $\pi$ is smaller than $R_t$  and we wait up to time $\rho_1$ when the process $(\gamma_t)_{t \geq 0}=(R_t)_{t \geq 0}$ hits the point $3\pi/2$. 
\item[b)]  For $ t \geq \rho_1$ we proceed as follows. On the probability space supporting the process $(r_t,\beta_t)_{t\geq 0}$ we can solve the stochastic differential equation $d\gamma_t=dB_t+c\sin(\gamma_t)\,dt$, where $(B_t)_{t\geq 0}$ is the Brownian motion driving the process  $(r_t,\beta_t)_{t\geq 0}$. Observe that the solution $(\gamma_t)_{t\geq 0}$ is pathwise unique. We wait up to the random time $\rho_2 = \rho_{2,1} \wedge \rho_{2,2}$, where 
\begin{displaymath}
\rho_{2,1} = \inf \lbrace t \mid \gamma_t = 2\pi\rbrace\,\text{   and    }\,\rho_{2,2} = \inf\lbrace t \mid \gamma_t - \pi = |\beta_t - \pi|\rbrace.
\end{displaymath}
Observe that up to the stopping time $\rho_2$ the distance of $\beta$ to $\pi$ does not exceed the distance of $\gamma$ to $\pi$.
\item[c)] If $\rho_2 = \rho_{2,1}$ then we extend $(\gamma_t)_t$ as a Brownian motion $(R^1_t)_{t \geq 0}$ at $2\pi$ which is reflected at $2\pi$ and wait until $\rho_{3,1}=\rho_{3,1,1}\wedge \rho_{3,1,2}$, where 
\begin{displaymath}
\rho_{3,1,1}= \inf\lbrace t \mid R^1_t = \pi \rbrace\,\text{   and   }\,\rho_{3,1,2}=\inf \lbrace t \mid R^1_t -\pi=| \beta_t - \pi|\rbrace.  
\end{displaymath}
If $\rho_2 = \rho_{2,2}$ then we extend the comparison process $(\gamma_t)_t$ as a reflecting Brownian motion $R^2_t$ if $\gamma_{\rho_{2,2}-} \in \lbrace 0 , 2\pi \rbrace$ or let it behave as a solution to $d\gamma_t=dB_t+c\sin(\gamma_t)\,dt$, and we wait up to the time $\rho_{3,2}=\rho_{3,2,1} \wedge \rho_{3,2,2}$, where  
\begin{equation*}
\begin{split}
\rho_{3,2,1} &= \inf\lbrace t \mid \gamma_t \in \lbrace 0, 2\pi \rbrace \text{  or  }R^2_t = \pi\rbrace\, \text{   and   }\\
\rho_{3,2,2}&=\inf\lbrace t\mid R^1_t -\pi=| \beta_t - \pi|\,(\gamma_t -\pi =| \beta_t - \pi|)\rbrace
\end{split}
\end{equation*}
meaning that we wait if we hit $\pi$ in the case that we move according to a reflected Brownian motion or $0,2\pi$ otherwise or alternatively until the comparison process has the same distance to $\pi$ as the process $\beta$ .

Observe once again that due to Theorem \ref{t:IW} we can conclude that the distance of the $\beta$-process to $\pi$ does not exceed the distance of the $\gamma$-process to $\pi$.  

Let us set $\rho_3=\rho_{3,1} \wedge \rho_{3,2}$.
\item[d)] At time $\rho_3$ there are three possibilities: If the $\beta$-process is at the boundary and the $\gamma$-process is in the interior, then we wait until either the distance of both processes to $\pi$ coincides or if the $\gamma$ hits the boundary. If both processes are at the boundary we can make use of our Lemmas and wait until the $\gamma$-process hits $\pi$ and if both have the same distance to $\pi$ we can use Theorem \ref{t:IW} until one process hits the boundary.  

In all cases the distance of the $\gamma$-process to $\pi$ dominates the distance of the $\beta$-process to $\pi$. Then this procedure is repeated in the same way. 

\end{itemize}
We emphasize, that the items a) -- d) do not allow to construct the coupling for all times.
Thus we have to observe that by construction of the coupling the process $(\gamma_t)_{t \geq 0}$ is always between the process $(\beta_t)_{t\geq 0}$ and the point $2\pi$ and we know, that the time $\sigma^1$ is finite for $\gamma_t$ and $\gamma_{\sigma^1}=\pi$. Therefore at time $\sigma^1$ we have $\beta_{\sigma^1}=\gamma_{\sigma^1}=\pi$ and we can start  again from item a) replacing $r_0$ by $r_{\sigma^1}$. 

Iteratively we obtain two processes: the process $(r_t,\beta_t)_{t \geq 0}$ and $(\gamma_t)_{t \geq 0}$ such that for all $t\geq 0$ we have 
\begin{displaymath}
\forall t \geq 0 : |\pi - \beta_t|\leq |\pi -\gamma_t|.
\end{displaymath}
Using this together with the fact that $[\pi, 2\pi] \ni l \mapsto \cos(|l|)$ is increasing it then becomes clear that $$r^{\gamma}_t \geq r_t.$$ Finally we conclude that 
\begin{displaymath}
\tau_R=\inf\lbrace t > 0 \mid r_t \leq R\bigr\rbrace \leq \tau^{\gamma}_R=\inf\lbrace t\geq 0 \mid r^{\gamma}_t\leq R\rbrace
\end{displaymath}
and since according to Lemma \ref{l:exponsimple} the random variable $\tau^{\gamma}_R$ has exponentially decaying tails the proof Proposition \ref{p:expintegrability} is completed.
\begin{remark}\label{r:final}
We would like to end with the following observation : The gradient form of the drift in \eqref{i:SDE} was dictated by the aim to have an explicit candidate for the invariant distribution. The apriori knowledge of the existence of an invariant distribution was used in Theorem \ref{t:mixing} as an essential ingredient. We want to emphasize that this is not true for Theorem \ref{t:main}.  For fixed $T>0$ the analog of the function $V_R$, defined in \eqref{e:Lfunction}, satisfies (see Theorem 6.2 in \cite{DMT95})
\begin{displaymath}
P_sV_R  \leq \lambda(s)V_R + b\,\mathbf{1}_{\overline{B(0,R)}\times [0,2\pi]},
\end{displaymath}
where $\lambda(s)$ is bounded for $s \in (0,T]$ with $\lambda(T)<1$ and $b \in (0,\infty)$. Therefore by Theorem 2.1 of \cite{DMT95} we can conclude the existence of an invariant distribution for the fiber lay down process. 

Thus one might hope that the approach presented in this work is also applicable to the case of a moving conveyer belt. A detailed investigation of this is deferred to the future. 
\end{remark}
\section*{Appendix}
In this section we recall Theorem 1.1 from chapter VI \cite{I-WBook}. Suppose we have two real continuous functions $b_1$ and $b_2$ defined on $[0,\infty)\times \mathbb{R}$ such that 
\begin{equation}\label{e:strict}
\forall t\geq 0,\,x\in\mathbb{R} : b_1(t,x)<b_2(t,x)
\end{equation}
and let $(\Omega,\mathcal{F},\mathbb{P})$ be a probability measure with a reference familiy $(\mathcal{F}_t)_{t\geq 0}$. 
\begin{theorem}\label{t:IW}
Suppose that we are given the following processes:
\begin{itemize}
\item[1)] two real $(\mathcal{F})_{t \geq 0}$-adapted continuous processes $x_1(t,\omega)$ and $x_2(t,\omega)$
\item[2)] a one-dimensional $(\mathcal{F}_t)_{t \geq 0}$-Brownian motion $B(t,\omega)$ such that $B_0=0$ a.s..
\item[3)] two real $(\mathcal{F}_t)_{t \geq 0}$-adapted well-measurable processes $\beta_1(t,\omega)$ and $\beta_2(t,\omega)$. 
\end{itemize}
We assume that they satisfy the following conditions with probability one:
\begin{equation}
x_i(t)-x_i(0)= B(t) + \int_0^t\beta_i(s)\,ds,\,i=1,2,
\end{equation}
\begin{equation}
x_1(0)\leq x_2(0)
\end{equation}
\begin{equation}
\beta_1(t)\leq b_1(t,x_1(t))\,\text{   for every $t\geq 0$}
\end{equation}
\begin{equation}
\beta_2(t)\geq b_2(t,x_2(t))\,\text{   for every $t\geq 0$}.
\end{equation}
Then with probability one, we have 
\begin{equation}\label{e:conclusion}
x_1(t) \leq x_2(t)\,\text{   for every $t\geq 0$}.
\end{equation}
Furthermore, if the pathwise uniqueness holds for at least one of the following stochastic differential equations
\begin{equation}
dX(t)=dB(t) + b_i(t,X(t))\,dt,\,i=1,2,
\end{equation}
then the same conclusion \eqref{e:conclusion} holds if \eqref{e:strict} is replaced by the weakened condition
\begin{equation}
b_1(t,x)\leq b_2(t,x)\,\text{  for  $t\geq0,\,x \in \mathbb{R}$}.
\end{equation}
\end{theorem}

\section*{Acknowledgements}
The first named author would like to thank Vitali Wachtel (Ludwig-Maximilians University Munich) for very useful advice and guidance concerning the theory of random walks as well as Martin Grothaus (University of Kaiserslautern) for introducing him to the process considered in this work as well as stimulating discussions.

\end{document}